\documentclass[11pt]{article}\textwidth 160mm\textheight 235mm
\usepackage{amsmath}
\usepackage{amsthm}
\usepackage{amsfonts}
\usepackage{multirow}
\usepackage[table,xcdraw]{xcolor}
\usepackage{subcaption}
\usepackage{enumerate}
\usepackage{float}
\usepackage{algorithm}
\usepackage{algpseudocode}
\usepackage{hyperref}
\usepackage{tabularx}
\usepackage{amssymb}
\usepackage{graphicx}
\usepackage{tcolorbox}
\usepackage{booktabs}
\usepackage[table,xcdraw]{xcolor}
\usepackage{multirow}
\usepackage{rotating,tabularx}

\usepackage[left=2cm,right=2cm,top=2cm,bottom=2cm]{geometry}
\newcolumntype{L}[1]{>{\raggedright\let\newline\\\arraybackslash\hspace{0pt}}m{#1}}
\newcolumntype{C}[1]{>{\centering\let\newline\\\arraybackslash\hspace{0pt}}m{#1}}
\newcolumntype{R}[1]{>{\raggedleft\let\newline\\\arraybackslash\hspace{0pt}}m{#1}}

\newtheorem{Theorem}{Theorem}[section]
\newtheorem{Proposition}[Theorem]{Proposition}
\newtheorem{Remark}[Theorem]{Remark}
\newtheorem{Lemma}[Theorem]{Lemma}
\newtheorem{Corollary}[Theorem]{Corollary}
\newtheorem{Definition}[Theorem]{Definition}
\newtheorem{Example}[Theorem]{Example}

\usepackage{hyperref}
\hypersetup{colorlinks=true,urlcolor=blue}
\expandafter\let\expandafter\oldproof\csname\string\proof\endcsname
\let\oldendproof\endproof
\renewenvironment{proof}[1][\proofname]{
\oldproof[\ttfamily\scshape \bf #1.]
}{\oldendproof}
\def\ve{\varepsilon}

\def\dom{{\rm dom}\,}

\def\epi{{\rm epi\,}}
\def\rge{{\rm rge\,}}

\def\ox{\overline{x}}

\def\tto{\rightrightarrows}

\def\Bar{\overline}
\def\ra{\rangle}
\def\la{\langle}
\def\ve{\varepsilon}
\def\epsilon{\varepsilon}
\def\ox{\bar{x}}

\def\ov{\bar{v}}

\def\co{\mbox{\rm co}\,}
\def\cone{\mbox{\rm cone}\,}

\def\rri{\rightrightarrows}

\def\gph{\mbox{\rm gph}\,}
\def\epi{\mbox{\rm epi}\,}

\def\dom{\mbox{\rm dom}\,}

\def\co{\mbox{\rm co}\,}

\def\argmin{\mathop{{\rm argmin}}}

\def\O{\Omega}

\def\st{\stackrel}
\def\oR{\Bar{\R}}
\def\lm{\lambda}

\def \N{{\rm I\!N}}
\def \R{{\rm I\!R}}

\newcommand{\bnab}{\overline{\nabla}}

\def\Limsup{\mathop{{\rm Lim}\,{\rm sup}}}
\def\Liminf{\mathop{{\rm Lim}\,{\rm inf}}}
\def\Lim{\mathop{{\rm Lim}}}

\def\Limsup{\mathop{{\rm Lim}\,{\rm sup}}}

\numberwithin{equation}{section}

\def\eliminf{\mathop{{\rm e} \text{-}\,{\rm liminf}}}
\def\elimsup{\mathop{{\rm e} \text{-}\,{\rm limsup}}}
\def\elim{\mathop{{\rm e} \text{-}\,{\rm lim}}}

\def\qua{\mathrm{quad}}

\numberwithin{equation}{section}
\title{\bf
Characterizations of Variational Convexity and Tilt Stability via Quadratic Bundles}

\author{Pham Duy Khanh\footnote{Group of Analysis and Applied Mathematics, Department of Mathematics, Ho Chi Minh City University of Education, Ho Chi Minh City, Vietnam. E-mail: khanhpd@hcmue.edu.vn. } \quad Boris S. Mordukhovich\footnote{Department of Mathematics, Wayne State University, Detroit, Michigan, USA. E-mail: aa1086@wayne.edu. Research of this author was partly supported by the US National Science Foundation under grant DMS-2204519, by the Australian Research Council under Discovery Project DP-190100555, and by Project 111 of China under grant D21024.}\quad Vo Thanh Phat\footnote{Department of Mathematics  and Statistics, University of North Dakota, Grand Forks, North Dakota, USA. E-mail: thanh.vo.1@und.edu.} \quad Le Duc Viet\footnote{Department of Mathematics, Wayne State University, Detroit, Michigan, USA. E-mail: vietle@wayne.edu. Research of this author was partly supported by the US National Science Foundation under grant DMS-2204519.}}

\begin{document}
\maketitle
\vspace*{-0.2in}
\begin{center}
{\bf Dedicated to the memory of Hedy Attouch,\\ an outstanding mathematician and very nice human being}
\end{center}

\noindent
{\small{\bf Abstract}. In this paper, we establish characterizations of variational $s$-convexity and tilt stability for prox-regular functions in the absence of subdifferential continuity via quadratic bundles, a kind of primal-dual generalized second-order derivatives recently introduced by Rockafellar. Deriving such characterizations in the effective pointbased form requires a certain revision of quadratic bundles investigated below. Our device is based on the notion of generalized twice differentiability and its novel characterization via classical twice differentiability of the associated Moreau envelopes combined with various limiting procedures for functions and sets.\\[1ex]
{\bf Mathematics Subject Classification (2020)}} 49J52, 49J53, 90C31

\section{Introduction}\label{intro}
\vspace*{-0.1in}

In the recent years, the notion of {\em $($strong$)$ variational convexity} of extended-real-valued functions whose subdifferential graphs are not distinguishable from those for locally (strongly) convex functions, emerged as an attempt to study local optimality and numerical algorithms of the proximal and augmented Lagrangian types; see the series of Rockafellar's papers \cite{varconv,roc,roc2,varprox}, which are mainly based on {\em second-order subderivative} constructions of the {\em primal type} in finite-dimensional spaces. Another {\em dual-type} approach to these fundamental issues, in both finite and infinite dimensions, was developed in \cite{KKMP23-2,var2order,varsuff1,Mordukhovich24} with applications \cite{kmp24} to generalized Newtonian methods for solving problems of nonsmooth nonconvex optimization and machine learning. The latter approach relays on the usage of coderivative-based {\em second-order subdifferentials} (or generalized Hessians) introduced by Mordukhovich in \cite{mord-2nd} and then largely developed in the literature over the years; see the new book \cite{Mordukhovich24} and the vast bibliographies therein for numerous theoretical and computational results involving such dual-type constructions with various applications to optimization, control, variational stability, and practical modeling. 

Quite recently \cite{roc}, Rockafellar suggested yet another, {\em primal-dual} approach to variational convexity and its strong counterpart based on the newly introduced notion of {\em quadratic bundles} for convex\footnote{When the results of our study in this paper and in the preceding one \cite{gtdquad} have been obtained, we got familiar with the very fresh Rockafellar's manuscript \cite{roc24} containing extensions of quadratic bundles to nonconvex functions and their applications to some issues discussed in \cite{gtdquad} and in the current paper; see more comments below.} extended-real-valued functions. This approach is associated with the concept of {\em epigraphical convergence} of sequences of functions and sets, which was originated by Wijsman in \cite{wij64} for convex functions and then largely grew outside of convexity. The fundamental contributions by Hedy Attouch into the study and applications of epi-convergence are difficult to overstate; see his monograph \cite{attouch84} as well as the book by Rockafellar and Wets \cite{Rockafellar98} with the references therein. 

The notion of quadratic bundles is closely related to {\em generalized twice differentiability} of extended-real-valued functions defined in \cite{roc} and then systematically studied in our preceding work \cite{gtdquad}, where this notion was equivalently characterized, for the general class of prox-regular functions, via the classical twice differentiability of the associated {\em Moreau envelopes}. The obtained characterization played a crucial role in \cite{gtdquad} for our investigation of generalized twice differentiability of prox-regular function and the modified version of quadratic bundles without subdifferential continuity. In particular, we established in \cite{gtdquad} the {\em nonemptiness} of quadratic bundles for the general class of prox-regular functions.

The main goal of this paper is to provide {\em pointbased primal-dual} characterizations of the notion of {\em variational $s$-convexity} of prox-regular functions via quadratic bundles in the {\em absence of subdifferential continuity}. In this way, we also obtain new pointbased primal-dual characterizations of {\em tilt-stable} local minimizers for such functions with quantitative {\em modulus relationships}. An important role in our analysis is played the newly established connections between the {\em quadratic bundle} of the function in question and the {\em Hessian bundle} of the associated Moreau envelope.

The rest of the paper is structured as follows. Section~\ref{sec:prel} equips the reader with some basic constructions and facts of variational analysis and generalized differentiation for sets, extended-real-valued functions and set-valued mappings.  Section~\ref{sec:recap} deals with generalized twice differentiability and quadratic bundles of functions. Section~\ref{sec:charvar} presents primal-dual characterization of variational $s$-convexity, while Section~\ref{sec:tilt} addresses tilt stability of local minimizers. The concluding Section~\ref{sec:conclusion} summarizes the main achievements of the paper and discusses some directions of our future research.

\section{Preliminaries from Variational Analysis}\label{sec:prel}

The setting of this paper is the $n$-dimensional space $\R^n$ equipped (unless otherwise stated) with the  Euclidean norm $\|\cdot\|$ and the inner product $\la \cdot, \cdot \ra$. We signify by $B(\bar{x},r)$ the open ball centered at $\bar{x}$ with radius $r > 0$ and write $x_k \xrightarrow{C} \bar{x}$ when the sequence $\{x_k\}\subset C$ converges to $\bar{x}$. For an extended-real-valued function $f: \R^n \to \R \cup \{\infty\} := \overline{\R}$, the \textit{domain} of $f$ and \textit{epigraph} of $f$ are defined, respectively, by
$$
\dom f := \big\{x \in \R^n\;\big|\; f(x) <\infty\big\}\;\mbox{ and }\;\epi f := \big\{(x,\alpha) \in \R^n \times \R\;\big|\;\alpha \ge f(x)\big\}.
$$
It is said that $f$ is \textit{proper} if $\dom f\ne\emptyset$ and $f$ is {\em lower semicontinuous} (l.s.c.) at $\bar{x}$ if $\liminf\limits_{x \to \bar{x}}  f(x) \ge f(\bar{x})$. The \textit{convex hull} and the \textit{conic convex conic hull} of a set $C\subset\R^n$ are denoted by $\co C$ and $\cone C$, respectively. The \textit{indicator function} of $C$ is defined by $\delta_C(x):=0$ for $x\in C$ and  $\delta_C(x):=\infty$ otherwise.

For a set-valued mapping/multifunction $F: \R^n \rightrightarrows \R^m$, the \textit{domain} and \textit{graph} of $F$ are
$$
\dom F :=\big\{x \in \R^n\;\big|\;F(x)\ne\emptyset\big\}\;\mbox{ and }\;\gph F :=\big\{(x,y) \in \R^n \times \R^m\;\big|\; y \in F(x)\big\}. 
$$
Given $f: \R^n \to \overline{\R}$, the (Fr\'echet) {\em regular subdifferential} of $f$ at $\bar{x}\in\dom f$ is defined by
$$
\widehat{\partial}f (\bar{x}) := \Big\{ v \in \R^n\;\Big|\; \liminf\limits_{x \to \bar{x}} \frac{f(x) - f(\bar{x}) - \la v, x - \bar{x} \ra}{\|x - \bar{x}\|} \ge 0\Big\} 
$$
and the (Mordukhovich) \textit{limiting subdifferential} of $f$ at $\bar{x}$ is 
\begin{equation}\label{eq:subdiff}
\partial f(\bar{x}) := \Limsup\limits_{x \xrightarrow{f} \bar{x}} \widehat{\partial} f(\bar{x}),
\end{equation}
where $\Limsup\limits_{x \to \bar{x}} F(x)$ stands for the (Painlev\'e-Kuratowski) {\em outer limit} of 
$F: \R^n \rri \R^m$ at $\bar{x}$ defined by
$$
\Limsup\limits_{x \to \bar{x}} F(x) :=\big\{ y \in \R^m\;\big|\;\exists\,x_k \to \bar{x},\;y_k \to y\;\mbox{ with }\;y_k \in F(x_k)\;\mbox{ as }\;k \in \N:=\{1,2,\ldots\}\Big\}.
$$
When $f$ is convex, both sets $\widehat{\partial} f$ and $\partial f$ agree with the subdifferential of convex analysis
$$
\partial f(\bar{x}) =\big\{ v \in \R^n\;\big|\;\la v, x - \bar{x} \ra \le f(x) - f(\bar{x})\;\mbox{ whenever }\;x \in \R^n\big\}. 
$$
When $f$ is nonconvex, the set $\widehat\partial f(\ox)$ is still convex, while $\partial f(\ox)$ is not so, even for simple functions as, e.g., $f(x)=-|x|$, where $\widehat\partial f(0)=\emptyset$ while $\partial f(0)=\{-1,1\}$. Nevertheless, $\partial f$ (but not $\widehat\partial f$) enjoys {\em full calculus} in both finite and infinite dimensions, which is based on {\em extremal/variational principles} of variational analysis; see the books \cite{Mordukhovich06,Mordukhovich18,Mordukhovich24,Rockafellar98,thibault} and the references therein for more details.\vspace*{0.05in} 

Next we briefly review the notions of \textit{epi-convergence} for sequences of extended-real-valued functions as well as of {\em second-order subderivatives} for such functions. More details on these and related notions of variational analysis can be found in \cite{Rockafellar98}. To begin with, recall that the {\em outer limit} and {\em inner limit} of a set sequence $\{C_k\}\subset\R^n$ are defined, respectively, by
$$
\Limsup\limits_{k \to \infty} C_k =\big\{ x\;\big|\;\exists\,\{x_{k_j}\}\subset\{x_k\}\;\mbox{ with }\;x_{k_j}\to x,\;x_k\in C_k\big\},\quad\Liminf\limits_{k \to \infty} C_k =\big\{ x\;\big|\;\exists\,x_k\st{C_k}{\to}x\big\}.
$$
If $\Limsup\limits_{x \to \infty} C_k = \Liminf\limits_{x \to \infty} C_k=:C$, then $\{C_k\}$ {\em converges} to $C=\Lim\limits_{x \to \infty}C_k$, which is written as $C_k \to C$.

For a sequence of functions $f_k : \R^n \to \overline{\R}$, the \textit{lower epigraphical limit} and the \textit{upper epigraphical limit} are defined respectively via their epigraphs as follows:
$$ \epi\Big( \eliminf\limits_{k \to \infty} f_k\Big) = \Limsup\limits_{k \to \infty} (\epi f_k),\quad\epi\Big( \elimsup\limits_{k \to \infty} f_k \Big) = \Liminf\limits_{k \to \infty} (\epi f_k).
$$
When these two limits agree, i.e., the sequence $\{\epi f_k\}$ converges to $\epi f$ for a function $f: \R^n \to \overline{\R}$, we say that $\{f_k\}$ \textit{epigraphically converges} to $f$, which is denoted by either $\elim\limits_{k \to \infty} f_k = f$, or $f_k \xrightarrow{e} f$. 

Alternative descriptions of the lower and upper epigraphical limits of functions are given in the following assertion taken from \cite[Proposition~7.2]{Rockafellar98}.

\begin{Proposition}\label{prop:epichar} Let $\{f_k\}$ be a sequence of extended-real-valued functions on $\R^n$, and let $x\in\R^n$. Then:
$$
\Big(\eliminf\limits_{k \to \infty} f_k\Big)(x) = \min\big\{ \alpha \in \overline{\mathbb{\R}}\;\big|\;\exists x_k \to x,\;\liminf f_k(x_k) = \alpha\big\}, 
$$
$$
\Big(\elimsup\limits_{k \to \infty} f_k\Big)(x) = \min\big\{ \alpha \in \overline{\mathbb{\R}}\;\big|\;\exists\,x_k \to x,\,\limsup f_k(x_k) = \alpha\big\}.
$$
Therefore, $\{f_k\}$ epigraphically converges to $f$ if and only if for each point $x \in \R^n$, we have
$$\left\{ \begin{aligned}
&\liminf f_k(x_k) \ge f(x)\;\text{ for every sequence }\;x_k \to x,\\
&\limsup f_k(x_k) \le f(x)\;\text{ for some sequence }\;x_k \to x.
\end{aligned} \right.
$$
In particular, if $\{f_k\}$ epigraphically converges to $f$, then for all $x\in\R^n$ there exists a sequence $x_k \to x$ such that $f_k(x_k) \to f(x)$ as $k\to\infty$.
\end{Proposition}

Using the notion of epi-convergence enables us to define, following Rockafellar \cite{2sd2}, the second-order subderivatives for extended-real-valued functions. Given $f\colon\R^n\to\oR$, $\ox\in\dom f$, $\ov\in\R^n$, and $t>0$, consider the \textit{second-order difference quotient} 
\begin{equation*}
\Delta_t^2 f(\bar{x}|\bar{v})(w) := \frac{f(\bar{x} + t w) - f(\bar{x}) - t \la \bar{v}, w \ra}{\frac{1}{2}t^2}, \quad  w \in \R^n,
\end{equation*}
and then define the \textit{second-order subderivative} of $f$ at $\bar{x}$ for $\bar{v}$ by
\begin{equation}\label{eq:d2}
d^2 f(\bar{x}|\bar{v})(w) := \Big(\eliminf\limits_{t \downarrow 0} \Delta_t^2 f(\bar{x}|\bar{v})\Big)(w) = \liminf\limits_{\substack{t \downarrow 0 \\ w' \to w}} \Delta_t^2 f(\bar{x}|\bar{v})(w'), \quad w\in \R^n.
\end{equation}
When the epigraphical lower limit in \eqref{eq:d2} is a full limit, it is said that $f$ is \textit{twice epi-differentiable} at $\bar{x}$ for $\bar{v}$. If in addition $d^2 f(\bar{x}|\bar{v})$ is proper, then $f$ is called to be \textit{properly twice epi-differentiable} at $\bar{x}$ for $\bar{v}$. Examples of nonsmooth twice 
epi-differentiable functions include, in particular, convex piecewise linear-quadratic functions \cite{Rockafellar98} and---much more general---{\em parabolically regular} functions as shown in \cite{modeb1,sarabi}. However, a similar formula to \eqref{eq:d2} for the upper epigraphical limit of $\Delta_t^2 f(\bar{x}|\bar{v})$ as $t \downarrow 0$ is not available, and we have to rely on Proposition~\ref{prop:epichar} giving us the expression
$$
\Big(\elimsup\limits_{t \downarrow 0} \Delta_t^2 f(\bar{x}|\bar{v})\Big)(w)  = \min\big\{ \alpha \in \overline{\mathbb{\R}}\;\big|\;\exists\,t_k \downarrow 0,\,w_k \to w,\,\limsup \Delta_{t_k}^2 f(\bar{x}|\bar{v})(w_k) = \alpha\big\}
$$
while verifying twice epi-differentiability. In the case of classical twice differentiable functions, second-order subderivatives reduce to quadratic forms associated with the Hessian matrix at $\bar{x}$:
\begin{equation}\label{eq:d2twicediff}
d^2 f(\bar{x}|\nabla f(\bar{x}))(w) = \la w, \nabla^2 f(\bar{x})w \ra, \quad  w \in \R^n. 
\end{equation}
We also need recalling two more interrelated notions, which play a crucial role in deriving our major results. Given a proper l.s.c.\ function $f\colon\R^n\to\oR$ and a parameter $\lm>0$, the \textit{Moreau $\lambda$-envelope} and the associated {\em $\lambda$-proximal mapping} of $f$  are defined by
$$
e_\lambda f(x) := \inf\limits_{u \in \R^n}\Big\{f(u) + \frac{1}{2\lambda}\|u - x\|^2\Big\},\quad
P_\lambda f(x) := \argmin\limits_{u \in \R^n}\Big\{f(u) + \frac{1}{2\lambda}\|u - x\|^2\Big\}.
$$
The function $f\colon\R^n\to\oR$ is called to be {\em prox-bounded} on $\R^n$ if there exists $\lm>0$ such that $e_\lm f(x)>-\infty$ for some $x\in\R^n$, which is equivalent to saying that $f$ is minorized by a quadratic function.\vspace*{0.05in}

The next notion of {\em variational $s$-convexity} for extended-real-valued functions is the main one we intend to characterize in this paper. For the general case of $s\in\R$, it has been very recently introduced by Rockafellar \cite{varprox} and have already been investigated in \cite{gfrerer2024,gtdquad,roc24}, while its previous versions with $s\ge 0$ appeared earlier in \cite{varconv} and then was studied and applied in \cite{var2order,varsuff1,kmp24,roc,roc2} in finite-dimensional spaces and in  \cite{KKMP23-2,Mordukhovich24} in infinite dimensions. The case of $s<0$ relates to {\em prox-regularity}, the notion introduced as early as in \cite{proxreg} and then systematically employed in variational analysis and optimization; see \cite{Rockafellar98,thibault}.

Before formulating the aforementioned variational $s$-convexity notion, recall that $f:\R^n \to \overline{\R}$ is {\em $s$-convex} for $s \in \R$ if  the quadratically $s$-shifted function $f - \frac{s}{2}\|\cdot\|^2$ is convex. When $s > 0$, the function $f$ is called to be \textit{$s$-strongly convex}, and it is \textit{$(-s)$-weakly convex} when $s < 0$. 

\begin{Definition}\label{defi:varconv} {\rm Let $s$ be any real number. A proper l.s.c.\ function $f\colon\R^n \to \overline{\R}$ is {\em variationally $s$-convex} at $\bar{x}$ for $\bar{v} \in \partial f(\bar{x})$ if $f$ is finite at $\bar{x}$ and there exist an $s$-convex function $\widehat{f}$ as well as convex neighborhoods $U$ of $\bar{x}$ and  $V$ of $\bar{v}$ together with $\varepsilon > 0$ such that $\widehat{f} \le f$ on $U$ and 
$$
(U \times V) \cap \gph \partial \widehat{f} = (U_\varepsilon \times V) \cap \gph \partial f, \quad\widehat{f}(x) = f(x)\;\text{ at common elements }\;(x,v), 
$$
where $U_\varepsilon := \{x \in U \mid f(x) < f(\bar{x}) + \varepsilon\}$. Note that we can always select $U := B(\bar{x},\varepsilon)$, $V := B(\bar{v},\varepsilon)$ and then say that $f$ is variationally $s$-convex at $\bar{x}$ for $\bar{v}$ \textit{with corresponding radius} $\varepsilon > 0$.} When $s \le 0$, $f$ is said to be $(-s)$-level {\em prox-regular} at $\bar{x}$ for $\bar{v}$, while for $s = 0$ we say that $f$ is {\em variationally convex}  at $\bar{x}$ for $\bar{v}$. When $s > 0$, $f$ is called to be {\em $s$-strongly variationally convex} at $\bar{x}$ for $\bar{v}$.
\end{Definition}

A useful notion to study variationally $s$-convex functions is the one of $f$-attentive localization of the subgradient set $\partial f$ around a reference point. To be specific, given a function $f: \R^n \to \overline{\R}$, its limiting subdifferential $\partial f$, and a number $\varepsilon > 0$,  the \textit{$f$-attentive $\varepsilon$-localization} $T_\ve\colon\R^n\tto\R^n$ of $\partial f$ around $(\bar{x},\bar{v}) \in \gph \partial f$ is defined via its graph by
$$ 
\gph T_\ve: =\big\{ (x,v) \in \gph \partial f\;\big|\;x \in B(\bar{x},\varepsilon),\,v \in B(\bar{v},\varepsilon),\;|f(x) - f(\bar{x})| < \varepsilon\big\}. 
$$
When $f$ is l.s.c.\ at $\bar{x}$, the above definition can be equivalently replaced by
\begin{equation}\label{eq:fattentive}
\gph T_\ve =\big\{ (x,v) \in \gph \partial f\;\big|\;x\in B(\bar{x},\varepsilon),\;v \in B(\bar{v},\varepsilon),\;f(x) < f(\bar{x}) + \varepsilon\big\}.
\end{equation}
The next result extracted from  \cite[Theorem~13.37]{Rockafellar98} presents an important property of Moreau envelopes and proximal mappings associated with prox-regular and prox-bounded functions.

\begin{Proposition}\label{prop:prox-mor-pr} Let $f: \R^n \to \overline{\R}$ be prox-bounded on $\R^n$ and $r$-level prox-regular at $\bar{x}$ for $\bar{v} \in \partial f(\bar{x})$ with the corresponding radius $\varepsilon > 0$. Then for all $\lambda \in (0,1/r)$, there exists a neighborhood of the shifted point $\bar{x} + \lambda\bar{v}$ on which we have the following properties:
\begin{itemize}
\item[\rm\textbf{(i)}] The proximal mapping $P_\lambda f$ is single-valued and Lipschitz continuous with $P_\lambda f(\bar{x} + \lambda\bar{v}) = \bar{x} + \lambda\bar{v}$.

\item[\rm\textbf{(ii)}] The Moreau envelope $e_\lambda f$ is of class $\mathcal{C}^{1,1}$ with 
\begin{equation}\label{eq:gradenv}
\nabla e_\lambda f = \lambda^{-1}[I - P_\lambda f] = [\lambda I + T^{-1}_\ve]^{-1},
\end{equation}
where $T_\ve$ is the $f$-attentive $\varepsilon$-localization of $\partial f$ at $(\bar{x},\bar{v})$.
\end{itemize}
Furthermore, the multifunction $T_\ve$ in {\rm \textbf{(ii)}} can be chosen so that the set $U_{\lambda} := \rge (I + \lambda T_\varepsilon)$ serves for all $\lambda > 0$ sufficiently small as a neighborhood of $\bar{x} + \lambda\bar{v}$ on which these properties hold.
\end{Proposition}

The function-attentive localization of the limiting subdifferential provides a characterization of variational 
$s$-convexity as it is shown in \cite[Theorem~1]{varprox}.

\begin{Proposition}\label{theo:svarchar} Let $f: \R^n \to \overline{\R}$ be l.s.c.\ at $\ox\in\dom f$, let $s\in\R$, and let $\ve>0$. Then  $f$ is variationally $s$-convex at $\bar{x}$ for $\bar{v} \in \partial f(\bar{x})$ if and only if there exist a neighborhood $U$ of $\bar{x}$ and an $f$-attentive $\ve$-localization $T_\ve$ of $\partial f$ around $(\bar{x},\bar{v})$ as in \eqref{eq:fattentive} such that for all $x' \in U$ we have the estimate
\begin{equation}\label{eq:svarconv}
f(x') \ge f(x) + \la v, x' - x \ra + \frac{s}{2}\|x' - x\|^2\;\mbox{ whenever }\;(x,v) \in \gph T_\ve.
\end{equation}
\end{Proposition}

A function $f\colon\R^n\to\oR$ is said to be \textit{subdifferentially continuous} at $\bar{x}\in\dom f$ for $\bar{v}\in\partial f(\ox)$ if for all $(x_k,v_k) \xrightarrow{\gph \partial f}(\bar{x},\bar{v})$ we have $f(x_k) \to f(\bar{x})$ as $k\to\infty$. The subdifferentially continuity of $f$ at $\bar{x}$ for $\bar{v}$ allows us to replace the set $\gph T_\ve$ above by the ordinary neighborhood of $(\bar{x},\bar{v})$ on $\gph \partial f$, i.e., the variationally $s$-convexity of $f$ at $\bar{x}$ for $\bar{v}$ holds if there exists $\varepsilon > 0$ such that
$$
f(x') \ge f(x) + \la v, x' - x \ra + \frac{s}{2}\|x' - x\|^2\;\mbox{ for all }\;x' \in B(\bar{x},\varepsilon)\;\mbox{ and }\;(x,v) \in \gph \partial f \cap B((\bar{x},\bar{v}),\varepsilon). 
$$

In what follows, we exploit the useful relationships between the second-order growth condition and second-order subderivative established in our preceding paper \cite[Theorem~3.4]{gtdquad}.

\begin{Proposition}\label{prop:strongchar} Let $f: \R^n \to \overline{\R}$ be a proper l.s.c.\ function. Pick arbitrary vectors $\bar{x} \in\dom f$ and $\bar{v}\in \R^n$ and consider the assertions:
 \begin{itemize}
\item[\rm\textbf{(i)}] There exists a neighborhood $U$ of $\bar{x}$ on which the second-order growth condition holds:
\begin{equation*}
f(x) \ge f(\bar{x}) + \la \bar{v}, x - \bar{x} \ra + \frac{\kappa}{2}\|x - \bar{x}\|^2\;\mbox{ for all }\;x\in U.
\end{equation*}

\item[\rm\textbf{(ii)}] We have the second-order subderivative estimate:
\begin{equation*}
d^2 f(\bar{x}|\bar{v})(w) \ge \mu\|w\|^2\;\mbox{ for all }\;w\in\R^n.
\end{equation*}
\end{itemize} 
Then implication {\rm\textbf{(i)}$\Longrightarrow$\textbf{(ii)}} is satisfied with $\mu = \kappa$, while {\rm\textbf{(ii)}$\Longrightarrow$\textbf{(i)}} holds with $\mu > \kappa$.
\end{Proposition}

Here is a direct consequence of Propositions~\ref{theo:svarchar} and \ref{prop:strongchar}, which is taken from \cite[Corollary 3.6]{gtdquad}.

\begin{Corollary}\label{cor:d2prox} Let $s\in\R$, and let $f: \R^n \to \overline{\R}$ be variationally $s$-convex at $\bar{x} \in \dom f$ for $\bar{v} \in \partial f(\bar{x})$. Then there exist $\ve>0$ and an $f$-attentive localization $T_\ve$ of $\partial f$ around $(\bar{x},\bar{v})$ such that 
\begin{equation*}
d^2 f(x|v)(w) \ge s\|w\|^2\;\mbox{ for all }\;w\in\R^n\;\mbox{ and }\;(x,v)\in \gph T_\ve.
\end{equation*}
\end{Corollary}

To quantify variational convexity, the {\em modulus of variational $s$-convexity} $\operatorname{cnv} f(\bar{x}|\bar{v})$ of an l.s.c.\ function $f\colon\R^n\to\oR$ at $\ox\in\dom f$ for $\bar{v} \in \partial f(\bar{x})$ was introduced in the recent preprints by Gfrerer \cite{gfrerer2024} and Rockafellar \cite{roc24} as the limit superior of all the real numbers $s$ in Definition~\ref{defi:varconv} when the neighborhood $U \times V$ of $(\bar{x},\bar{v})$ shrinks to the singleton $\{(\bar{x},\bar{v})\}$ and $\varepsilon \downarrow 0$. A negative real value of this quantity indicates prox-regularity while a nonnegative value indicates (strong) variational convexity. Several quantitative expressions for $\operatorname{cnv} f(\bar{x}|\bar{v})$ were obtained in \cite{gfrerer2024,roc24}. In what follows, we'll use the one taken from \cite{roc24}:
\begin{equation}\label{def:cnv}
\operatorname{cnv} f(\bar{x}|\bar{v}) = \limsup\limits_{\substack{ \beta \downarrow 0\\ (x,v) \to (\bar{x},\bar{v})\\f(x) \to f(\bar{x})}}\Bigg( \min\limits_{\substack{\|w\| = 1 \\ \tau \in (0, \beta]}} \Delta_\tau^2 f(x|v)(w)\Bigg).
\end{equation}

\section{Generalized Twice Differentiability and Quadratic Bundles}\label{sec:recap}

In this section, we recall the notions of \textit{generalized twice differentiability} and \textit{quadratic bundles} of functions while providing essential facts that will be useful for deriving the main results of this paper in the subsequent sections. More details and the additional material can be found in our preceding paper \cite{gtdquad}.

Given a matrix $A \in \R^{n \times n}$, denote the {\em quadratic form} associated with $A$ by
\begin{equation}\label{qua}
q_A(x) := \la x, Ax \ra\;\mbox{ for all }\;x \in \R^n.
\end{equation}
The next proposition presents a generalized version of the Cauchy-Schwartz inequality.

\begin{Proposition} Let $A\in\R^{n\times n}$ be a symmetric and positive-definite matrix, and let $q_A$ be defined in \eqref{qua}. Then we have the generalized Cauchy-Schwartz inequality 
\begin{equation}\label{eq:genCS}
q_A(x) + q_{A^{-1}}(y) \ge 2 \la x, y \ra\;\mbox{ for all }\; x,y\in\R^n.
\end{equation}
\end{Proposition}
\begin{proof}
Denote by $\lambda_1,\ldots,\lambda_n > 0$ the eigenvalues of $A$. By performing the standard diagonalization, we find an orthogonal matrix $P$ (i.e., $P^T = P^{-1})$ and $D: = \mathrm{diag} (\lambda_1,\ldots,\lambda_n)$ such that $A = PDP^{T}$. Then
$$
A^{-1} = (P^T)^{-1}D^{-1}P^{-1} = PD^{-1}P^T,
$$
which brings us to the relationships
$$
\begin{aligned}
q_{A}(x) + q_{A^{-1}}(y) &= \la x, PDP^{T}x \ra + \la y, PDP^{T}y \ra \\
&= \la P^T x, DP^Tx \ra + \la P^T y, D^{-1}P^Ty \ra\\
&= \sum\limits_{i=1}^n \left( \lambda_i (P^T x)_i^2 + \lambda_i^{-1}(P^T y)_i^2 \right)\\
&\ge 2\sum\limits_{i=1}^n (P^T x)_i(P^T y)_i =  2 \la P^Tx, P^Ty \ra =  2\la x,y \ra
\end{aligned}
$$
and thus verifies the claimed inequality \eqref{eq:genCS}.
\end{proof} 

An extension of the quadratic form \eqref{qua} is provided below. 

\begin{Definition}\rm A function $q: \R^n \to \overline{\R}$ is called a \textit{generalized quadratic form} if it is expressible as $q = q_A + \delta_L$, where $L$ is a linear subspace of $\R^n$, and where $A \in \R^{n \times n}$ is a symmetric matrix. 
\end{Definition}

Generalized quadratic forms are crucial for defining generalized twice differentiability of extended-real-valued functions due to Rockafellar \cite{roc}.

\begin{Definition}\label{defi:quaddiff} \rm 
A function $f: \R^n \to \overline{\R}$ is \textit{generalized twice differentiable} at $\bar{x}$ for $\bar{v} \in \partial f(\bar{x})$ if it is twice epi-differentiable at $\bar{x}$ for $\bar{v}$ and if the second-order subderivative $d^2 f(\bar{x}|\bar{v})$ is a generalized quadratic form.
\end{Definition}

The next result, taken from \cite[Proposition~4.8]{gtdquad}, helps us to determine the generalized twice differentiability of functions under summation.

\begin{Proposition}\label{prop:sumC2gendif} Let $g: \R^n \to \overline{\R}$ be l.s.c.\ at $\ox\in\dom f$ with $\bar{v} \in \partial g(\bar{x})$, and let $f: \R^n \to \overline{\R}$  be strictly differentiable at $\bar{x}$ and twice differentiable at this point. Then the following hold:
\begin{itemize}
\item[\rm\textbf{(i)}] We always have the equality
$$
d^2 (f+g)(\bar{x}|\nabla f(\bar{x}) + \bar{v})(w) = \la w, \nabla^2 f(\bar{x})w \ra + d^2 g(\bar{x}|\bar{v})(w)\;\mbox{ for all }\;w\in\R^n.
$$
\item[\rm\textbf{(ii)}] If $g$ is generalized twice differentiable at $\bar{x}$ for $\bar{v}$, then the summation function $f + g$ is generalized twice differentiable at $\bar{x}$ for $\nabla f(\bar{x}) + \bar{v}$.
\end{itemize}
\end{Proposition}

The following proposition, established in \cite[Lemma~5.3]{gtdquad}, provides a limiting correspondence between the prox-regular function in question and the associated Moreau envelope.

\begin{Proposition}\label{lem:zkvkxk} Assume that $f: \R^n \to \overline{\R}$ is prox-regular at $\bar{x}$ for $\bar{v} \in \partial f(\bar{x})$ and fix $\lambda > 0$ so small that all the assertions in 
Proposition~{\rm\ref{prop:prox-mor-pr}} hold; in particular, $e_{\lambda}f$ is of class $\mathcal{C}^{1,1}$ around $\bar{z}:= \bar{x} + \lambda \bar{v}$. Given a sequence $z_k \to \bar{z}$, define the vectors
\begin{equation}\label{eq:vkxk}
 v_k := \nabla  e_{\lambda} f(z_k), \quad x_k := z_k - \lambda v_k.   
\end{equation}
Then we have $v_k \in \partial f(x_k)$ for all large $k\in\N$ and
$$
(x_k,v_k) \to (\bar{x},\bar{v}), \quad f(x_k) \to f(\bar{x})\;\mbox{ as }\;k\to\infty.
$$
\end{Proposition}

The next result, taken from \cite[Theorem~5.8 and Corollary~5.9]{gtdquad}, reduces the generalized twice differentiability of prox-regular functions to the classical twice differentiability of Moreau envelopes while playing a crucial role in our approach to characterize variational convexity and tilt stability.

\begin{Proposition}\label{cor:gtdtwicediff} Given $r>0$, assume that $f: \R^n \to \overline{\R}$ is prox-bounded on $\R^n$ and  $r$-level prox-regular at $\bar{x}\in\dom f$ for $\bar{v}\in \partial f(\bar{x})$. Then there exists an $f$-attentive $\ve$-localization $T_\ve$ of $\partial f$ around $(\bar{x},\bar{v})$ for any $\ve>0$ sufficiently small, such that the following are equivalent whenever $\lambda \in (0,1/r)$ and $(x,v) \in \gph T_\ve$:
\begin{itemize}
\item[{\rm\textbf{(i)}}] $f$ is generalized twice differentiable at $x$ for $v$. 

\item[{\rm\textbf{(ii)}}] $e_{\lambda} f$ is twice differentiable at $x + \lambda v$.
\end{itemize}
Moreover, we have the relationship
\begin{equation}\label{eq:envd^22}
 e_{\lambda}\Big[ \frac{1}{2}d^2 f (x|v)\Big] =  d^2\Big[ \frac{1}{2} e_{\lambda} f\Big]  \left(x + \lambda v \Big| \dfrac{1}{2}v \right.\Big)\;\mbox{ for all }\;(x,v) \in \gph T_\ve.
\end{equation}
\end{Proposition}

Now we are ready to recall the definition of \textit{quadratic bundles} for extended-real-valued functions, which was proposed by Rockafellar \cite{roc} for convex functions and then quite recently extended in \cite{gtdquad,roc24} to nonconvex settings. Given $f\colon\R^n\to\oR$, consider the set
\begin{equation}\label{Qf}
\Omega_f :=\big\{(x,v) \in \gph \partial f\;\big|\;f\;\text{ is generalized twice differentiable at }\;x \;\text{ for }\; v \big\}.
\end{equation}

\begin{Definition}\label{defi:quadbund} \rm The \textit{quadratic bundle} $\mathrm{quad} f (\bar{x}|\bar{v})$ of $f$ at $\bar{x} \in \dom f$ for $\bar{v} \in \partial f(\bar{x})$ is defined as the collection of generalized quadratic forms $q$ for which there exists $(x_k,v_k) \xrightarrow{\Omega_f} (\bar{x},\bar{v})$ such that $f(x_k) \to f(\bar{x})$ and the sequence of generalized quadratic forms $q_k = \frac{1}{2} d^2 f(x_k|v_k)$ converges epigraphically to $q$ as $k\to\infty$.
\end{Definition}

As we see from the definition, the quadratic  bundle is a {\em primal-dual} second-order generalized derivative, which involves limiting procedures over primal-dual pairs. If $f$ is subdifferentially continuous at $\bar{x}$ for $\bar{v}$, then the requirement $f(x_k) \to f(\bar{x})$ is superfluous. The following example shows that there can be a significant difference between the revised construction and the original construction denoted by $\qua_o f(\bar{x}|\bar{v})$. To be specific, the notation $\qua_o f(\bar{x}|\bar{v})$ stands for the collection of generalized quadratic forms $q$ for which there exists $(x_k,v_k) \xrightarrow{\Omega_f} (\bar{x},\bar{v})$ such that the sequence of generalized quadratic forms $q_k = \frac{1}{2} d^2 f(x_k|v_k)$ converges epigraphically to $q$ as $k\to\infty$.

\begin{Example}\label{exam:quadandquads} \rm Define the function $f: \R \to \R$ by
$$
f(x) = \left\{\begin{array}{ll}
x^2,  & x \ge 0,  \\
1,  & x < 0.
\end{array} \right. 
$$
Then $\partial f(0) = (-\infty,0]$, $f$ is {\em variationally $2$-strongly convex} at $0$ for $0$ (therefore, it is prox-regular at $0$ for $0$) while {\em not being subdifferentially continuous} at $0$ for $0$. We have here
\begin{equation}\label{eq:quadandquads}
\qua_o f(0|0) = \{ q_{[0]},  q_{[1]}, \delta_{\{0\}}\}, \quad \qua f(0|0) = \{q_{[1]}, \delta_{\{0\}}\}.
\end{equation}
\end{Example} 

\begin{Remark}\label{prop:extractepi} \rm The observations  below clarify the construction of quadratic bundles. 
\begin{itemize}
\item[\textbf{(i)}] The requirements imposed on the sequence $\{(x_k,v_k)\}$
in Definition~\ref{defi:quadbund} {\em automatically yield} the epigraphical convergence of the functions $\dfrac{1}{2}d^2 f(x_k|v_k)$. Indeed, it follows from the structure of the set $\O_f$ in \eqref{Qf} that the choice of $\{(x_k,v_k)\}$  such that $(x_k,v_k) \xrightarrow{\Omega_f} (\bar{x},\bar{v})$ and  $f(x_k) \to f(\bar{x})$ allows us to assume without loss of generality that $\dfrac{1}{2}d^2 f(x_k|v_k)$ epigraphically converge to some function without any additional information. In other words, the construction of quadratic bundles relies {\em solely} on the sequences $\{(x_k,v_k)\}$ that fulfill the imposed requirements.

\item[\textbf{(ii)}] If $f:\R^n \to \overline{\R}$ is prox-regular at $\bar{x}\in\dom f$ for $\bar{v} \in \partial f(\bar{x})$, then for all $(x_k,v_k) \xrightarrow{\Omega_f} (\bar{x},\bar{v})$ and $f(x_k) \to f(\bar{x})$ we have that the epigraphical convergence of the functions $\frac{1}{2}d^2 f(x_k|v_k)$ ensures that the limiting function $q$ is a generalized quadratic form. In other words, $\mathrm{quad} f (\bar{x}|\bar{v})$ accumulates {\em all the 
epi-limits} of the functions $\dfrac{1}{2}d^2 f(x_k|v_k)$ when $(x_k,v_k) \xrightarrow{\Omega_f} (\bar{x},\bar{v})$ and $f(x_k) \to f(\bar{x})$ as $k\to\infty$.
\end{itemize}
\end{Remark}
The next result provides a {\em sum rule} for quadratic bundles employed below; see  
\cite[Theorem~6.6]{gtdquad}. 

\begin{Proposition}\label{prop:sumrulequad} Let $g: \R^n \to \overline{\R}$ be l.s.c.\ at $\ox\in\dom g$ for $\bar{v} \in \partial g(\bar{x})$, and let $f: \R^n \to \overline{\R}$ be $\mathcal{C}^2$-smooth around $\bar{x}$. Then $\nabla f(\bar{x}) + \bar{v} \in \partial (f+g)(\bar{x})$ and
\begin{equation*}
\qua (f+g)(\bar{x}|\nabla f(\bar{x}) + \bar{v}) = \frac{1}{2} q_{\nabla^2 f(\bar{x})} + \qua g(\bar{x} | \bar{v}). 
\end{equation*}
In particular, we have the representations
$$
\qua f(\bar{x}|\nabla f(\bar{x})) = \Big\{ \frac{1}{2} d^2 f(\bar{x}|\nabla f (\bar{x}))\Big\} =\Big\{ \frac{1}{2} q_{\nabla^2 f(\bar{x})}\Big\}.
$$
\end{Proposition}

Finally, we present here one of the principal results of \cite[Theorem~6.11]{gtdquad}, which lays the foundation of the characterizations of variational of variational convexity and tilt stability established below by showing the {\em existence} of quadratic bundles used in  these characterizations for prox-regular functions.

\begin{Proposition}\label{prop:quadproxne} Let $f: \R^n \to \overline{\R}$ be prox-regular at $\bar{x}\in\dom f$ for $\bar{v} \in \partial f(\bar{x})$. Then the quadratic bundle $\qua f(\bar{x}|\bar{v})$ is nonempty.
\end{Proposition}

\section{Variational $s$-Convexity via Quadratic Bundles}\label{sec:charvar}

The main result of this section establishes a primal-dual characterization of variational $s$-convexity of {\em prox-regular} functions $f\colon\R^n\to\oR$ via quadratic bundles {\em without assuming subdifferential continuity}. We'll see in the proof below that for this broad class of extended-valued functions, it is sufficient to consider the case of $s>0$, i.e., to confine ourselves to characterizing {\em strong variational convexity}. Note that the subdifferential continuity requirement has been recently dismissed by Gfrerer \cite{gfrerer2024} in {\em dual} characterizations of strong variational convexity with the usage of $f$-attentive modification of Mordukhovich's second-order subdifferentials in the previously known results. Now we develop novel {\em pointwise primal-dual} characterizations of this property by utilizing quadratic bundles; see also Remark~\ref{rem:rock}(ii). Our approach is based on characterizing first the property of {\em strong $s$-convexity} for functions of class ${\cal C}^{1,1}$ via Hessian bundles and then propagating this device to quadratic bundle characterizations of {\em variational $s$-convexity} of prox-regular functions via the reduction to Moreau envelopes of class ${\cal C}^{1,1}$.\vspace*{0.05in}

Recall that the \textit{Hessian bundle} of a function $f\colon\R^n\to\oR$ at a point $\bar{x} \in \dom f$ is defined by
\begin{equation*}
\overline{\nabla}^2 f(\bar{x}) :=\big\{ H\in\R^{n\times n}\;\big|\;\exists\,x_k \to \bar{x}\;\text{ s.t. }\;f \;\text{ is twice differentiable at }\;x_k\;\text{ and }\;\nabla^2 f(x_k) \to H \big\}.
\end{equation*}
It follows from \cite[Theorem~13.52]{Rockafellar98} that Hessian bundles are nonempty for functions of class ${\cal C}^{1,1}$. The next result of its own interest establishes relationships between strong convexity of ${\cal C}^{1,1}$ functions and positive-definiteness of Hessian bundles with the {\em modulus interplay}.

\begin{Theorem}\label{prop:conv}
Assume that $f: \R^n \to\oR$ is of class $\mathcal{C}^{1,1}$ around $\bar{x}$ and consider the following assertions:
\begin{itemize} 
\item[\rm\textbf{(i)}] There exists a neighborhood $U$ of $\bar{x}$ such that $f$ is strongly convex on $U$ with modulus $s > 0$.

\item[\rm\textbf{(ii)}] Every $H\in \bnab^2 f(\bar{x})$ is positive-definite with modulus $\mu > 0$.

\item[\rm\textbf{(iii)}] There exists a neighborhood $U$ of $\bar{x}$ such that for any $x \in U$, every $H \in \bnab^2 f(x)$ is positive-definite with modulus $s > 0$.
\end{itemize}
Then we have that implication {\rm \textbf{(i)}}$\Longrightarrow$ {\rm\textbf{(ii)}} holds when $\mu=s$, that implication {\rm \textbf{(ii)}}$\Longrightarrow$ {\rm \textbf{(iii)}} holds when $\mu >s$, while implication
{\rm\textbf{(iii)}}$\Longrightarrow$ {\rm \textbf{(i)}} is always fulfilled.
\end{Theorem}
\begin{proof} Assuming in \textbf{(i)} that $f$ is strongly convex with modulus $s>0$ and is of class $\mathcal{C}^{1,1}$ in a neighborhood $U$ of $\bar{x}$ gives us by \cite[Theorem~4.1.1]{urruty} the estimate
$$
f(x') \ge f(x) + \la \nabla f(x), x' - x \ra + \frac{s}{2}\|x' - x\|^2\;\mbox{ for all }\;x,x\in U.$$
Then it follows from Proposition~\ref{prop:strongchar} that
\begin{equation}\label{eq:d2posdefk}
d^2 f(x|\nabla f(x))(w) \ge s\|w\|^2\;\mbox{ whenever }\;x\in U,\;w\in\R^n.
\end{equation}
Pick any $H \in \overline{\nabla}^2 f(\bar{x})$ and find $x_k \to \bar{x}$ such that $f$ is twice differentiable at $x_k$ and $\nabla^2 f(x_k) \to H$ as $k\to\infty$. We get $d^2 f(x_k|\nabla f(x_k)) = q_{\nabla^2 f(x_k)}$ and thus deduce from \eqref{eq:d2posdefk} with $x=x_k$ that
$$
\la w, \nabla^2 f(x_k) w \ra \ge s\|w\|^2\;\mbox{ for all }\;w\in\R^n\;\mbox{ and large }\;k.
$$
Passing to the limits as $k \to\infty$ brings us to 
$$
\la w, H w\ra \ge s \|w\|^2\;\mbox{ whenever }\;w\in\R^n,
$$
which therefore justifies \textbf{(ii)} is satisfied for $\mu =s$.\vspace*{0.05in}

Suppose now that \textbf{(ii)} holds with $\mu >s$, i.e.,
\begin{equation*}
\la w, Hw \ra \ge \mu\|w\|^2\;\mbox{ for all }\;H \in \overline{\nabla}^2 f(\bar{x}) \;\mbox{ and }\;w\in\R^n.
\end{equation*}
The latter easily implies that
\begin{equation}\label{eq:hesposdef}
\la w, Mw \ra \ge \mu\|w\|^2\;\mbox{ for all }\; M \in \co\big[\overline{\nabla}^2 f(\bar{x})\big],\quad w\in\R^n. 
\end{equation}
since $\overline{\nabla}^2 f(x)$ consists of only symmetric matrices by \cite[Theorem~13.52]{Rockafellar98}; the same holds  for the set $\co[\overline{\nabla}^2 f(x)]$ if $x$ is sufficiently close to $\bar{x}$. Let us show that there is a neighborhood $U$ of $\bar{x}$ such that
\begin{equation}\label{eq:oscHes}
\mbox{ for all }\;x \in U,\; M_1 \in \co \left[\overline{\nabla}^2 f(x)\right]\;\mbox{ there is }\; M_0 \in \co \left[\overline{\nabla}^2 f(\bar{x})\right]\;\text{ with }\;\sigma(M_1 - M_0) < \mu -s, 
\end{equation}
where $\sigma(A)$ denotes the \textit{spectral norm} of a symmetric matrix $A$, i.e., the greatest among the absolute values of the eigenvalues of $A$. Suppose on the contrary that there there is no such $U$. It follows from the convexity of the set $\co[ \overline{\nabla}^2 f(\bar{x})] + \varepsilon\mathbb{B}$ for any $\varepsilon > 0$ and the unit norm $\mathbb{B}$ associated with the norm $\sigma$ that there exists a sequence $x_k \to \bar{x}$ such that for each $k$ we can find $M_k \in\overline{\nabla}^2 f(x_k)$ satisfying $\sigma (M_k - M) \ge \mu -s$ whenever $M \in \co[\overline{\nabla}^2 f(\bar{x})]$. Since $f$ is of class $\mathcal{C}^{1,1}$ around $\bar{x}$, the multifunction $x\mapsto\co[ \overline{\nabla}^2 f]$ is locally bounded around $\bar{x}$, which enables us to extract a subsequence of $\{M_k\}$ converging to some $M_0$. It follows therefore that $\sigma (M_0 - M) \ge \mu - s> 0$ for all $M \in \co[ \overline{\nabla}^2 f(\bar{x})]$, i.e., $M_0 \notin\co[\overline{\nabla}^2 f(\bar{x})]$.

We now show that $M_0 \in\overline{\nabla}^2 f(\bar{x})$ to get a contradiction. Indeed, the inclusion $M_k \in  \overline{\nabla}^2 f(x_k)$ for each $k\in\N$ allows us to find a sequence $\{x_k^m\}_m$ such that $x_k^m \to x_k$ and $\nabla^2 f(x_k^m) \to M_k$ as $m\to\infty$, while $f$ is twice differentiable at $x_k^m$. This yields the existence of $m(k) \in \N$  for each $k$ satisfying the condition
$$
\sigma( \nabla^2 f(x_k^m) - M_k) < \dfrac{1}{k}\;\mbox{ whenever }\;m > m(k).
$$
Denote $m_k := \max \{k, m(k)\}$ and $y_k := x_k^{m_k}$ for each $k\in\N$ and then get that $y_k \to \bar{x}$, that $f$ is twice differentiable at $y_k$, and that $\nabla^2 f(y_k) \to M_0$ as $k\to\infty$. This tells us that $M_0 \in   \overline{\nabla}^2 f(\bar{x})$, and thus the existence of a neighborhood $U$ of $\bar{x}$ satisfying \eqref{eq:oscHes} is justified.

Our next step is to check that, within the obtained neighborhood $U$, we have the estimate
\begin{equation}\label{eq:posdefheslan}
\la w, Mw \ra \ge s\|w\|^2\;\mbox{ for all }\;x \in U,\; w \in \R^n,\;\mbox{ and }\;M \in \co\big[\overline{\nabla}^2 f(x)\big].
\end{equation} 
To proceed, pick $x \in U$, $M \in \co \left[\overline{\nabla}^2 f(x)\right]$ and find by \eqref{eq:oscHes} a matrix $M_0 \in \co[\overline{\nabla}^2 f(\bar{x})]$ with $\sigma (M - M_0) < \mu - s$. This means that all the eigenvalues of $M - M_0$ are greater than $-(\mu -s)$, which yields
\begin{equation}\label{eq:hesposdef2}
\la w, (M - M_0) w \ra \ge -(\mu -s)\|w\|^2,\quad w\in\R^n.
\end{equation}
Since $M_0 \in \co[\overline{\nabla}^2 f(\bar{x})]$, it follows from \eqref{eq:hesposdef} that
\begin{equation}\label{eq:hesposdef3}
\la w, M_0w \ra \ge \mu\|w\|^2\;\mbox{ for all }\;w\in\R^n.
\end{equation}
Combining \eqref{eq:hesposdef2} and \eqref{eq:hesposdef3} ensures that
\begin{equation*}
\la w, Mw \ra \ge \mu \|w\|^2 + (-\mu +s)\|w\|^2 =s\|w\|^2,\quad w\in\R^n,
\end{equation*}
which verifies that $U$ is indeed a neighborhood of $\bar{x}$ that satisfies \eqref{eq:posdefheslan}. Observing that symmetric positive-definite matrices with modulus $s>0$ forms a convex set justifies the fulfillment of \textbf{(iii)}.

Assume finally that \textbf{(iii)} is satisfied and show that \textbf{(i)} always holds. Indeed, fixing any $x, u \in U$ and applying \cite[Theorem~2.3]{urrutyhessian} give us $x_u \in [x,u] \subset U$ and $M \in \co[\overline{\nabla}^2 f(x_u)]$ such that
$$
f(x) = f(u) + \la \nabla f(u), x - u \ra + \dfrac{1}{2}\la M(b - a),b - a \ra.
$$
By \textbf{(iii)} we have the estimate in \eqref{eq:posdefheslan}, which implies in turn that
$$
f(x) \ge f(u) + \la \nabla f(u), x - u \ra + \dfrac{s}{2}\|x - u\|^2. 
$$
It follows from \cite[Theorem~4.1.1]{urruty} that $f$ is strongly convex with modulus $s$ on $U$, which verifies \textbf{(i)} and thus completes the proof of the theorem.
\end{proof}

Preparing to prove the main result below, we need---besides Theorem~\ref{prop:conv} above---the following two technical lemmas. The first one is taken from \cite[Lemma~5.7]{gtdquad}.

\begin{Lemma}\label{lem:posdef} Let $L$ be a subspace of $\R^n$, let  $A: \R^n \to \R^n$ be a self-adjoint linear operator, and let $\sigma \in \R$. Impose on $L$ the eigenvalue lower-bound
$$
\la w, Aw \ra \ge \sigma \|w\|^2\;\mbox{ for all }\;w \in L.
$$
Then there exists a self-adjoint linear operator $B: \R^n \to \R^n$ such that the  eigenvalue lower-bound of $A$ on $L$ is extended to the same property for $B$ on the entire space, i.e.,
$$\la w,Bw \ra = \la w, Aw \ra\;\mbox{ for all }\;w\in L,
$$
$$\la w, Bw \ra \ge \sigma \|w\|^2 \;\mbox{ for all }\;w\in\R^n.
$$
\end{Lemma}

The next lemma reflects the stability of quadratic estimates under epi-limits.

\begin{Lemma}\label{lem:d2posdef} Let $f_k : \R^n \to \overline{\R}$ be a sequence of functions that are l.s.c.\ and positively homogeneous of degree $2$. Assume that $\eliminf\limits_{k \to \infty} f_k = f$ for a function $f: \R^n \to \overline{\R}$ satisfying
\begin{equation}\label{eq:bdtq}
f(w) \ge \mu\|w\|^2\;\;\mbox{ for all }\;w\in\R^n 
\end{equation}
with some $\mu \in \R$. Then for any $\delta > 0$ and $k\in\N$ sufficiently large, we have  
\begin{equation}\label{eq:bdtd2}
f_k(w) \ge (\mu - \delta)\|w\|^2\;\mbox{ whenever }\;w\in\R^n.
 \end{equation}
\end{Lemma}
\begin{proof} Since $f = \eliminf\limits_{k \to \infty} f_k$ and all $f_k$ are l.s.c.\ and positively homogeneous of degree $2$, it follows that $f$ is also l.s.c.\ and positively homogeneous of the same degree. Since both sides of \eqref{eq:bdtd2} are positively homogeneous of degree $2$, we only need to show that \eqref{eq:bdtd2} holds for all $w$ such that $\|w\| = 1$ whenever $k$ is sufficiently large. The compactness of the set $\{w\in\R^n \mid \|w\| = 1\}$  and \cite[Proposition~7.29(a)]{Rockafellar98} combined with \eqref{eq:bdtq} bring us to the estimates
\begin{equation*}
\liminf\limits_{k \to \infty} \left(\min\limits_{\|w\| = 1} f_k(w) \right) \ge  \min\limits_{\|w\| = 1} f(w) \ge \mu.
\end{equation*}
Hence we arrive at $\min\limits_{\|w\| = 1}f_k(w) \ge \mu - \delta$, which obviously implies \eqref{eq:bdtd2} and thus completes the proof.
\end{proof}

Now we are ready to establish the pointbased primal-dual characterization for variational $s$-convexity via the revised quadratic bundles for prox-regular functions. 

\begin{Theorem}\label{theo:varcharquad} Let $f: \R^n \to \overline{\R}$ be an l.s.c.\ function with $\bar{x} \in \dom f$, and let $s\in\R$. Suppose that $f$ is prox-regular at $\bar{x}$ for $\bar{v} \in \partial f(\bar{x})$ and consider the following assertions:
\begin{itemize}
\item[\bf\textbf{(i)}] $f$ is variationally $s$-convex at $\bar{x}$ for $\bar{v}$.

\item[\bf\textbf{(ii)}] The quadratic bundle $\qua f(\bar{x}|\bar{v})$  satisfies the uniform eigenvalue lower-bound
\begin{equation}\label{quad_posi}
q(w) \ge \mu\|w\|^2\;\mbox{ for all }\;q \in \qua f(\bar{x}|\bar{v}),\quad w \in \R^n. 
\end{equation}
\end{itemize} 
Then implication {\rm\textbf{(i)} $\Longrightarrow$\textbf{(ii)}} holds with $\mu = \dfrac{s}{2}$, while the reverse implication {\rm \textbf{(ii)}$\Longrightarrow$ \textbf{(i)}} is satisfied with any number $\mu > \dfrac{s}{2}$.
\end{Theorem}
\begin{proof} 
It is sufficient to consider only in the case where $s > 0$. Indeed, assume that $f$ is prox-regular at $\bar{x}$ for $\bar{v}$ with level $r \ge 0$ and define the function $g: \R^n \to \overline{\R}$  by
\begin{equation}\label{eq:temp1}
g(x) := f(x) + r \|x - \bar{x}\|^2,\quad x\in\R^2.   
\end{equation}
It is straightforward to see from Definition~\ref{defi:varconv} that $g$ is variationally $r$-convex. Furthermore, it follows from \eqref{eq:temp1} and Proposition~\ref{prop:sumrulequad} that
$$
\qua g(\bar{x}|\bar{v}) = \qua f(\bar{x}|\bar{v}) + r \|\cdot\|^2.
$$
Therefore, $\qua f (\bar{x}|\bar{v})$ is uniformly positive-definite with modulus $\mu$ if and only if $\qua g(\bar{x}|\bar{v})$ is uniformly positive-definite with modulus $\mu + r$. This shows that it is perfectly viable to work with the shifted function $g$ in \eqref{eq:temp1} and then return to $f$.

Having the above in mind, we proceed with verifying \textbf{(i)}$\Longrightarrow$ \textbf{(ii)} for $s>0$. When $f$ is variationally $s$-convex at $\bar{x}$ for $\bar{v} \in \partial f(\bar{x})$ with modulus $s > 0$, it follows from Proposition~\ref{theo:svarchar} that there exist a neighborhood $U$ of $\bar{x}$ and an $f$-attentive localization $T_\ve$ of $\partial f$ around $(\bar{x},\bar{v})$ such that 
$$f(x) \ge f(u) + \la v, x - u \ra + \frac{s}{2}\|x - u\|^2\;\mbox{ for all }\;(u,v) \in \gph T_\ve\;\mbox{ and }\;x\in U. 
$$
If necessary, we can shrink $\gph T_\ve$ so that $U$ is a neighborhood of $u$ whenever $(u,v) \in \gph T_\ve$. Applying Proposition~\ref{prop:strongchar} for each $(u,v) \in \gph T_\ve$ tells us that
\begin{equation}\label{quad-d2}
d^2 f(u|v)(w) \ge s\|w\|^2\;\mbox{ for all }\;w\in\R^n.
\end{equation}
Pick any $q \in \qua f(\bar{x}|\bar{v})$ and find a sequence $(x_k, v_k) \xrightarrow{\gph \partial f} (\bar{x},\bar{v})$ such that $f(x_k) \to f(\bar{x})$ and $\dfrac{1}{2} d^2 f(x_k|v_k) \xrightarrow{e} q $ as $k\to\infty$, where the function $f$ is twice epi-differentiable at $x_k$ for $v_k$, $k\in\N$, and where $\dfrac{1}{2}d^2 f(x_k|v_k)$ is a generalized quadratic form. It follows from Proposition~\ref{prop:epichar} that for each $w \in \R^n$ there exists a sequence $w_k \to w$ along which
$$
\lim\limits_{k \to \infty} \frac{1}{2} d^2 f(x_k|v_k)(w_k) = q(w).
$$
Since $(x_k, v_k) \xrightarrow{\gph \partial f} (\bar{x},\bar{v})$ and $f(x_k) \to f(\bar{x})$ as $k\to\infty$, we get $(x_k,v_k) \in \gph T_\ve$ for all $k$ sufficient large. Applying \eqref{quad-d2} for such $k$ yields the estimates
$$
\frac{1}{2}d^2 f(x_k|v_k)(w_k) \ge \frac{s}{2}\|w_k\|^2.
$$
Passing to the limit as $k \to \infty$ brings us to
$$
q(w) = \lim\limits_{k \to \infty} \frac{1}{2} d^2 f(x_k|v_k)(w_k) \ge  \lim\limits_{k \to \infty} \frac{s}{2}\|w_k\|^2 = \frac{s}{2}\|w\|^2,
$$
which therefore justifies assertion \textbf{(ii)}.\vspace*{0.05in}

Given any $s>0$, next we verify implication \textbf{(ii)} $\Longrightarrow$\textbf{(i)}. Suppose without loss of generality that $f$ is $s$-level prox-regular at $\bar{x}$ for $\bar{v}$ and that the quadratic bundle $\qua f(\bar{x}|\bar{v})$ is uniformly positive-definite with modulus $\mu >\dfrac{s}{2}$. To prove that $f$ is variationally $s$-convex at $\bar{x}$ for $\bar{v}$, fix any $\lambda \in (0,1/r)$ and rely on \cite[Theorem~4.4]{varsuff1} telling us that the variational $s$-convexity of the function $f$ at $\bar{x}$ for $\bar{v}$ is equivalent to the local strong convexity of the associated Moreau envelope $e_\lambda f$ around $\bar{x} + \lambda \bar{v}$ with modulus $\dfrac{s}{1 +s\lambda}$. By Proposition~\ref{prop:prox-mor-pr}, the Moreau envelope $e_{\lambda} f$ is of class $\mathcal{C}^{1,1}$ around $\bar{x} + \lambda \bar{v}$. Appealing to the Hessian bundle characterizations of Theorem~\ref{prop:conv}, we intend to show that every $H \in \overline{\nabla}^2 e_\lambda f (\bar{x} + \lambda \bar{v})$ is positive-definite with some modulus $\sigma > \dfrac{s}{1 +s\lambda}$. Indeed, for any $H \in \overline{\nabla}^2 e_\lambda f (\bar{x} + \lambda \bar{v})$ there is a sequence $z_k \to \bar{x} + \lambda \bar{v}$ such that $e_\lambda f$ is twice differentiable at $z_k$ and $\nabla^2 e_{\lambda} f(z_k) \to H$. Define the sequences $\{v_k\}$ and $\{x_k\}$ by
$$
v_k := \nabla e_{\lambda}f (z_k), \quad x_k := z_k - \lambda v_k,\quad k\in\N, 
$$
and deduce from Proposition~\ref{lem:zkvkxk} that $(x_k,v_k) \xrightarrow{\gph \partial f} (\bar{x},\bar{v})$ and $f(x_k) \to f(\bar{x})$ as $k\to\infty$. It follows from Proposition~\ref{cor:gtdtwicediff} by the twice differentiability of $e_\lambda f$ at $z_k$ that $f$ is generalized twice differentiable at $x_k$ for $v_k$ whenever $k\in\N$. This ensures that $(x_k,v_k) \xrightarrow{\Omega_f} (\bar{x},\bar{v})$ and $f(x_k) \to f(\bar{x})$. Invoking now Remark~\ref{prop:extractepi}{\bf(ii)}, we can extract a subsequence $(x_{k_m},v_{k_m}) \xrightarrow{\Omega_f} (\bar{x},\bar{v})$ such that $f(x_{k_m}) \to f(\bar{x})$ and $\dfrac{1}{2} d^2 f(x_{k_m}|v_{k_m}) \xrightarrow{e} q$ as $m\to\infty$, where $q$ is a generalized quadratic form. Obviously, $q \in \qua f(\bar{x}|\bar{v})$ by Definition~\ref{defi:quadbund}, and we can assume without loss of generality that $\frac{1}{2} d^2f (x_k | v_k) \xrightarrow{e} q$ as $k\to\infty$.
Applying \eqref{quad_posi} yields $q(w) \ge \mu \|w\|^2$ for all $w\in \R^n$. Employing further Lemma~\ref{lem:d2posdef} and the condition $\mu > \dfrac{s}{2}$ tells us that
\begin{equation}\label{eq:d2posdef} 
d^2 f(x_k|v_k)(w) \ge \left(\dfrac{s}{2} + \mu \right)\|w\|^2,\quad w\in\R^n, 
\end{equation}
whenever $k$ is sufficiently large. Since $f$ is generalized twice differentiable at $x_k$ for $v_k$, there exists a symmetric matrix $M_k$ and a linear subspace $L_k$ satisfying
\begin{equation}\label{eq:d2prox1}
 d^2 f(x_k|v_k) = \dfrac{1}{2}q_{M_k} + \delta_{L_k} 
\end{equation}
for each large $k$. Combining the latter with \eqref{eq:d2posdef}, we arrive at
$$
\frac{1}{2}\la w, M_k w \ra \ge \left(\dfrac{s}{2} + \mu \right)\|w\|^2\;\mbox{ whenever }\;w\in L.
$$
Lemma~\ref{lem:posdef} allows us to find a positive-definite symmetric matrix $A_k$ such that  
\begin{equation}\label{eq:Ak}
\frac{1}{2} \la w, A_k w \ra \ge \left(\dfrac{s}{2} + \mu  \right) \|w\|^2,\quad w\in\R^n,
\end{equation}
and that $\la w, A_k w \ra = \la w, M_k w \ra$ for all $w \in L_k$. Using this together with \eqref{eq:d2prox1} gives us
\begin{equation}\label{eq:d2prox2}
d^2 f(x_k|v_k) = \dfrac{1}{2}q_{A_k} + \delta_{L_k}.  
\end{equation}
Representation \eqref{eq:d2twicediff} applies as $e_\lambda f$ is twice differentiable at $z_k$ for all large $k$. Then 
Proposition~\ref{cor:gtdtwicediff} and the relationship between second-order subderivatives of $f$ and $e_\lambda f$ in \eqref{eq:envd^22} combined with \eqref{eq:d2prox2} yield
$$
\begin{aligned}
\dfrac{1}{2}\la w,\nabla^2 e_\lambda f(z_k) w \ra = d^2\Big[\dfrac{1}{2}e_{\lambda}f\Big]\Big(z_k\Big|\dfrac{1}{2}v_k\Big)(w)
&=e_{\lambda}\Big(\dfrac{1}{2}d^2 f(x_k|v_k)\Big)(w)\\
&= \dfrac{1}{2}\inf\limits_{x \in \R^n}\Big\{\frac{1}{2}\la x,A_kx \ra + \delta_{L_k}(x) + \frac{1}{\lambda}\|x - w\|^2\Big\}\\
&\ge \dfrac{1}{2} \inf\limits_{ x \in \R^n}\Big\{\frac{1}{2}\la x,A_kx \ra  + \frac{1}{\lambda}\|x - w\|^2\Big\}.
\end{aligned}
$$
Note that for each $k$, the matrix $\dfrac{\lambda}{2}\left(A_k + \dfrac{2}{\lambda}I\right)$ is 
positive-definite, so we can employ \eqref{eq:genCS} to get
$$\begin{aligned}
& \frac{1}{2}\la x,A_kx \ra  + \frac{1}{\lambda}\|x - w\|^2 \\
= \ & \frac{1}{\lambda} \left\la x, \dfrac{\lambda}{2}\left(A_k + \dfrac{2}{\lambda}I\right)x \right\ra + \dfrac{1}{\lambda}\|w\|^2 - \dfrac{2}{\lambda}\la x, w \ra  \\
\ge \ & \frac{1}{\lambda} \left\la x, \dfrac{\lambda}{2}\left(A_k + \dfrac{2}{\lambda}I\right)x \right\ra + \dfrac{1}{\lambda}\|w\|^2 - \dfrac{1}{\lambda} \left\la x, \dfrac{\lambda}{2}\left(A_k + \dfrac{2}{\lambda}I\right)x \right\ra - \dfrac{1}{\lambda} \left\la w, \dfrac{2}{\lambda}  \left(A_k + \dfrac{2}{\lambda}I\right)^{-1}w  \right\ra  \\
= \ &\dfrac{1}{\lambda}\|w\|^2 - \dfrac{2}{\lambda^2} \left\la w, \left(A_k + \dfrac{2}{\lambda}I\right)^{-1}w  \right\ra.
\end{aligned} 
$$

By \eqref{eq:Ak}, all the eigenvalues of $A_k$ are at least $\left(s + 2\mu  \right)$. It then follows that all the (positive) eigenvalues of $\left( A_k + \dfrac{2}{\lambda}I \right)^{-1}$ are at most $\dfrac{1}{s + 2\mu + 2/\lambda}$. Hence for large $k$ and all $w$, we have
$$
\begin{aligned}
\la w, \nabla^2 e_\lambda f(z_k) w \ra &\ge - \dfrac{2}{\lambda^2} \left\la w, \left(A_k + \dfrac{2}{\lambda}I\right)^{-1}w  \right\ra + \dfrac{1}{\lambda}\|w\|^2 \\
&\ge -\frac{2}{\lambda^2} \cdot\dfrac{1}{s+ 2\mu +2/\lambda}\|w\|^2 + \frac{1}{\lambda}\|w\|^2\\
&= -\frac{1}{\lambda} \cdot\dfrac{1}{\lambda(s/2 + \mu) + 1}\|w\|^2 + \frac{1}{\lambda}\|w\|^2\\
&= \frac{1}{\lambda} \cdot \left( 1 - \dfrac{1}{\lambda\left(s/2 + \mu\right) + 1}\right)\|w\|^2\\
&= \frac{1}{\lambda} \cdot \frac{\lambda (s/2 + \mu)}{\lambda (s/2 + \mu) + 1}\|w\|^2 = \frac{s/2 + \mu}{\lambda (s/2 + \mu) + 1}\|w\|^2.
\end{aligned} $$
Passing the limit $k \to \infty$, it follows that
$$
\begin{aligned}
\la w, H w \ra \ge \frac{s/2 + \mu}{\lambda (s/2 + \mu) + 1}\|w\|^2\;\mbox{ for all }\;w\in\R^n\setminus\{0\}
\end{aligned}
$$
Since the function $h(x):= \dfrac{x}{\lambda x + 1}$ is strictly increasing on $(0,\infty)$ and since $s/2 + \mu > s$, the matrix $H$ is positive-definite with modulus $\dfrac{s/2 + \mu}{\lambda (s/2 + \mu) + 1} > \dfrac{s}{\lambda s + 1}$. This completes the proof of the theorem.
\end{proof}

\begin{Remark}\label{rem:rock}$\,$\rm
\begin{itemize}
\item[\textbf{(i)}] The suggested revision of the quadratic bundle definition is {\em necessary} in order to obtain the result as in Theorem \ref{theo:varcharquad}. Indeed, consider the function $f$ from 
Example~\ref{exam:quadandquads}, where $\qua_o f(\bar{x}|\bar{v})$ denotes the quadratic bundle of $f$ at $\bar{x}$ for $\bar{v}$ before the revision. It is obvious that not all the quadratic forms in $\qua_o f(0|0)$ are
positive-definite, while $f$ is strongly variationally convex at $0$ for $0$.

\item[\textbf{(ii)}] After completing the draft of this paper, we observed the modulus formula for strong convexity of prox-regular functions derived in Theorem~4.5 of the very recent Rockafellar's preprint \cite{roc24} via quadratic bundles. Although these results are closely interrelated, our proof is thoroughly different from \cite{roc24}. Indeed, instead of establishing and exploiting the connection between the Hessian bundle of Moreau envelope and the quadratic bundle of the original function, the proof of \cite{roc24} utilizes the connection between the modulus $\operatorname{cnv} f(\bar{x}|\bar{v})$ as in \eqref{def:cnv} and the Lipschitz modulus of the derivative of the Fenchel conjugate for variationally strongly convex functions. 
\end{itemize}
\end{Remark}

We conclude this section with the following {\em pointbased necessary condition} for variational (not strong) convexity of l.s.c.\ via the positive-semidefiniteness of quadratic bundles.

\begin{Corollary} Let $f: \R^n \to \overline{\R}$ be an l.s.c.\ function with $\bar{x} \in \dom f$. If $f$ is variationally convex at $\bar{x}$ for $\bar{v}\in\partial f(\ox)$, then the quadratic bundle $\qua f(\bar{x}|\bar{v})$ is positive-semidefinite in the sense that $q \ge 0$ for all $q\in \qua f(\bar{x}|\bar{v})$. In particular, if $f$ is $\mathcal{C}^1$-smooth around $\bar{x}$, then the local convexity of $f$ yields the positive semidefiniteness of $\qua f(\bar{x}|\nabla f(\bar{x}))$.
\end{Corollary}
\begin{proof} Assuming that $f$ is variationally convex at $\bar{x}$ for $\bar{v}$ gives us by definition that the function $g:= f + \|\cdot\|^2$ is variationally $2$-strongly convex at $\bar{x}$ for $\bar{v} + 2\bar{x}$. In this case, $f$ is prox-regular at $\bar{x}$ for $\bar{v}$, which entails the prox-regularity of $g$ at $\bar{x}$ for $\bar{v} + 2\bar{x}$. Theorem~\ref{theo:varcharquad} tells us that the quadratic bundle $\qua g(\bar{x}|\bar{v} + 2\bar{x})$ is uniformly positive-definite with modulus $1$, i.e.,
\begin{equation}\label{eq:varconvness1}
q(w) \ge \|w\|^2\;\mbox{ for all }\;q \in \qua g(\bar{x}|\bar{v}),\quad  w \in \R^n,
\end{equation}
It follows from the sum rule in Proposition~\ref{prop:sumrulequad} that
\begin{equation}\label{eq:varconvness2}
\qua g(\bar{x}|\bar{v} + 2\bar{x}) = \|\cdot\|^2 +  \qua f(\bar{x}|\bar{v}). 
\end{equation}
Combining \eqref{eq:varconvness1} and \eqref{eq:varconvness2} yields 
$$ 
q(w) + \|w\|^2 \ge \|w\|^2\;\mbox{ whenever }\;q \in \qua f(\bar{x}|\bar{v}),\quad  w \in \R^n, 
$$
and thus verifies the claimed necessary condition. If $f$ is $\mathcal{C}^1$-smooth around $\bar{x}$, then the variational convexity of $f$ at $\bar{x}$ for $\nabla f(\bar{x})$ reduces to the local convexity of $f$ around $\bar{x}$, which completes the proof.
\end{proof}

\section{Tilt Stability of Local Minimizers}\label{sec:tilt}

This section is devoted to another important topic of variational analysis and optimization, which---in contrast of variational stability---directly addresses local minimizers of functions while happens to be closely related to the {\em strong} version of variational convexity at stationary points. We are talking about {\em tilt stability} of local minimizers, the notion introduced by Poliquin and Rockafellar \cite{tilt} as follows.

\begin{Definition}\label{tilt}\rm Given a function $f: \R^n \to \overline{\R}$, it is said that $\bar{x}\in\dom f$ is a \textit{tilt-stable local minimizer} of $f$ with modulus $\kappa > 0$ if there exists $\delta > 0$ such that the mapping
$$
M \colon v \mapsto \argmin\limits_{\|x - \bar{x}\| < \delta}\big\{ f(x) - f(\bar{x}) - \la v, x - \bar{x} \ra\big\} 
$$
is single-valued and Lipschitz continuous with modulus $\kappa$ on some neighborhood of $v = 0$ with $M(0) = \bar{x}$.
\end{Definition}

The first pointbased characterization of tilt stable local minimizers for prox-regular and subdifferentially continuous functions in finite dimensions was established in \cite{tilt}, without modulus involved, via the positive-definiteness of the {\em second-order subdifferential} mapping in the sense of \cite{mord-2nd}. Another approach to such {\em dual} characterizations of tilt stability in finite and infinite dimensions was developed in \cite{mnghia,tiltnghia}, where neighborhood characterizations were also derived in Hilbert spaces with modulus formulas; see also \cite{dima}. Explicit dual characterizations of tilt-stable minimizers in (constrained) problems of nonlinear and conic programming were obtained in \cite{gfrerer,mor15,tiltrocka}, etc. We refer the reader to the book \cite{Mordukhovich24} and the bibliographies therein for various results in this direction based on second-order subdifferentials of prox-regular and subdifferentially continuous functions and their calculations in finite and infinite dimensions. Quite recently \cite{gfrerer2024}, the subdifferential continuity was dismissed from both neighborhood and pointbased characterizations of tilt stable minimizers in the extended-real-valued unconstrained format of Definition~\ref{tilt} by incorporating the $f$-attentive convergence in the second-order subdifferential construction. 

{\em Primal-dual neighborhood} characterizations  of tilt stable minimizers for prox-regular and subdifferentially continuous functions on $\R^n$ were first obtained in \cite{tiltgraphical} in terms of {\em subgradient graphical derivatives}. It was later shown in \cite{nghia.24} that, in the case of {\em twice epi-differentiable} functions, the obtained characterizations can be equivalently expressed via the {\em second-order subderivatives} from \cite{2sd2,Rockafellar98}.\vspace*{0.05in}

Now we establish novel {\em primal-dual pointbased} characterizations of tilt-stable minimizers of general prox-regular functions {\em without subdifferential regularity} in terms of {\em quadratic bundles}. It is shown in \cite[Proposition~2.9]{varsuff1} that, for prox-regular and subdifferentially continuous functions $f\colon\R^n\to\oR$, the tilt-stability of a local minimizer $\ox$ of $f$ is {\em equivalent} to the variational strong stability of $f$ at $\ox$ for $\ov=0\in\partial f(\ox)$, while Example~3.4 from \cite{lewis} demonstrates that the imposed subdifferential continuity assumption on $f$ at its stationary points $\ox$ is {\em essential} for the fulfillment of this equivalence. 

On the other hand, it follows the dual tilt stability characterizations in \cite[Theorem~5.1]{gfrerer2024} that including the $f$-attentive convergence in the modified definition of the second-order subdifferentials provides the same characterizations of tilt-stable minimizers of prox-regular functions without subdifferential continuity at stationary points as those for variational strong convexity. Since the construction of  quadratic bundles in Definition~\ref{defi:quadbund} incorporates the $f$-attentive convergence, we can obtain the desired primal-dual characterizations of tilt stability for prox-regular functions in the absence of subdifferential continuity.

\begin{Theorem}\label{cor:tiltquad} Let $f: \R^n \to \overline{\R}$ be a proper function that is prox-regular function at $\bar{x} \in \dom f$ for $0\in \partial f(\bar{x})$. Consider the following assertions:
\begin{itemize}
\item[\bf\textbf{(i)}] $\bar{x}$ is a tilt-stable local minimizer of $f$ with modulus $\kappa > 0$.

\item[\bf\textbf{(ii)}] The quadratic bundle $\qua f(\bar{x}|0)$ satisfies the uniform eigenvalue lower-bound with $\mu > 0$ as in \eqref{quad_posi},  in which case we say that $\qua f(\bar{x}|0)$  is uniformly positive-definite with modulus $\mu > 0$.
\end{itemize} 
Then we have {\rm\textbf{(i)}} $\Longrightarrow${\rm\textbf{(ii)}} with $\mu = \dfrac{1}{2\kappa}$, while {\rm\textbf{(ii)}}$\Longrightarrow${\rm \textbf{(i)}} holds with any $\mu > \dfrac{1}{2\kappa}$.
\end{Theorem}
\begin{proof}
It follows from \cite[Corollary~5.2]{gfrerer2024} that for any function that is l.s.c.\ at $\ox\in\dom f$ with $0\in\partial f(\ox)$ and for any $\kappa>0$, the tilt stability of $f$ at $\ox$ with modulus $\kappa^{-1}$ is equivalent to the variational $\kappa$-convexity of $f$ at $\ox$ for $\ov=0$. Applying the characterization of  variational strong convexity established in Theorem~\ref{theo:varcharquad} verifies the claimed pointbased  characterizations of tilt-stable minimizers.
\end{proof}

\section{Conclusions and Future Research}\label{sec:conclusion}

This paper employs the fresh notion of quadratic bundles for extended-real-valued functions to derive pointbased primal-dual characterizations of variational $s$-strong convexity and tilt stability of local minimizers for the general class of prox-regular functions without the subdifferential continuity requirement. The major role in our device is played establishing novel relationships between the generalized variational properties in questions for the broad classes of extended-real-valued functions under consideration and the corresponding classical properties and constructions for the associated Moreau envelopes.

Topics of our future research include the study and characterizations via generalized bundles of the plain (not strong) property of variational convexity ($s=0$) for functions that may not be prox-regular and thus the variational $s$-convexity for them does not reduce to the case where $s>0$. We also aim at investigating infinite-dimensional extensions  of the obtained and newly proposed results.

\end{document}